\documentclass[10pt]{article}
\usepackage{amsmath,amsthm,amsfonts,latexsym,amsopn,verbatim,amscd,amssymb}
\theoremstyle{plain}
\newtheorem{theorem}{Theorem}[section]
\newtheorem{corollary}[theorem]{Corollary}
\newtheorem{definition}[theorem]{Definition}
\newtheorem{proposition}[theorem]{Proposition}

\newcommand{\bnum}{\begin{enumerate}}
\newcommand{\enum}{\end{enumerate}}

\numberwithin{equation}{section}

\DeclareMathOperator{\Aut}{Aut}
\DeclareMathOperator{\Inn}{Inn}
\DeclareMathOperator{\orb}{orb}
\DeclareMathOperator{\cl}{cl}
\begin{document}

\title{\textbf{On relative autocommutativity degree of a subgroup of a finite group}}
\author{Parama Dutta and Rajat Kanti Nath\footnote{Corresponding author}}
\date{}
\maketitle
\begin{center}\small{\it
Department of Mathematical Sciences, Tezpur University,\\ Napaam-784028, Sonitpur, Assam, India.\\



Emails:\, parama@gonitsora.com and rajatkantinath@yahoo.com}
\end{center}

\medskip

\begin{abstract}
In this paper, we consider the probability that a randomly chosen automorphism of a finite group fixes a randomly chosen element of a subgroup of that group. We obtain several new results as well as  generalizations and improvements of some existing results on this probability.
\end{abstract}

\medskip

\noindent {\small{\textit{Key words:} Automorphism group, Autocommuting probability, Autoisoclinism.}


\medskip

\section{Introduction}
Let $H$ be a subgroup of a finite group $G$ and $\Aut(G)$ be the automorphism group of $G$. The relative autocommutativity degree of $H$ denoted by ${\Pr}(H,\Aut(G))$ is the probability that a randomly chosen automorphism of $G$ fixes a randomly chosen element of $H$. In other words
\begin{equation}\label{GenAutComDeg}
{\Pr}(H,\Aut(G)) = \frac {\left|\{(x, \alpha)\in H \times \Aut(G) :  \alpha(x) = x\}\right|}{|H||\Aut(G)|}.
\end{equation}
The notion of ${\Pr}(H,\Aut(G))$ was introduced in \cite{moga} and studied in \cite{moga,Rismanchian15}. Note that ${\Pr}(G,\Aut(G))$ is the probability that an automorphism of $G$ fixes an element of it. The ratio ${\Pr}(G,\Aut(G))$ is also known as autocommutativity degree of $G$. It is worth mentioning that the study of autocommutativity degree of a finite group was initiated by Sherman  \cite{sherman}, in the year 1975. 
In this paper, we obtain several new results on ${\Pr}(H,\Aut(G))$ including some generalizations and improvements of existing results. 

For any element $x \in G$ and $\alpha \in \Aut(G)$ we write $[x,\alpha] := x^{-1}\alpha (x)$,  the autocommutator of $x$ and $\alpha$. We also write $S(H,\Aut(G)) := \{[x,\alpha] : x \in H \text{ and } \alpha \in \Aut(G)\}$, $L(H,\Aut(G)):= \{x \in H : \alpha(x) = x \text{ for all }\alpha \in \Aut(G)\}$ and $[H,\Aut(G)] := \langle S(H,\Aut(G)) \rangle$.  Note that $L(H,\Aut(G))$ is a normal subgroup of $H$ contained in  $H \cap Z(G)$ and $L(H,\Aut(G)) = \underset{\alpha \in \Aut(G)}{\cap}C_H(\alpha)$, where $Z(G)$ is the center of $G$ and   $C_H(\alpha) := \{x\in H:\alpha(x) = x\}$ is a subgroup of $H$. If $H = G$ then $L(H,\Aut(G)) = L(G)$, the absolute centre of $G$ (see \cite{hegarty}). Let $C_{\Aut(G)}(x) := \{\alpha\in \Aut(G) : \alpha(x) = x\}$ for  $x \in H$ and  $C_{\Aut(G)}(H) := \{\alpha \in \Aut(G) : \alpha (x)  = x \text{ for all } x \in H\}$. Then $C_{\Aut(G)}(x)$ is a subgroup of $\Aut(G)$ and  $C_{\Aut(G)}(H) = \underset{x\in H}{\cap}C_{\Aut(G)}(x)$.

%
%
%
%
It is easy to see that
\[
\{(x, \alpha)\in H \times \Aut(G) : \alpha(x) = x\} = \underset{x\in H}{\sqcup}(\{x\} \times C_{\Aut(G)}(x)) = \underset{\alpha \in \Aut(G)}{\sqcup}(C_H(\alpha) \times \{\alpha\}),
\]
where $\sqcup$ stands for disjoint union of sets. Hence
\begin{equation}\label{formula2}
|H||\Aut(G)|{\Pr}(H,\Aut(G)) = \underset{x\in H}{\sum}|C_{\Aut(G)}(x)| =  \underset{\alpha \in \Aut(G)}{\sum}|C_H(\alpha)|.
\end{equation}
Also $\Aut(G)$ acts on $G$   by the action $(\alpha,x)\mapsto \alpha(x)$  for $\alpha \in \Aut(G)$ and $x\in G$.  Let $\orb(x) := \{\alpha(x) : \alpha \in \Aut(G)\}$ be the orbit of $x \in G$. Then by orbit-stabilizer theorem, we have
$
|\orb(x)| = |\Aut(G)|/|C_{\Aut(G)}(x)|
$
and hence, \eqref{formula2} gives the following generalization of \cite[Proposition 2]{aK16}
\begin{equation}\label{grpdefn}
  {\Pr}(H,\Aut(G))  =   \frac {1}{|H|}\underset{x\in H}{\sum}\frac{1}{|\orb(x)|} = \frac {|\orb(H)|}{|H|}
\end{equation}
where $\orb(H) = \{\orb(x) : x \in H\}$.

 Note that ${\Pr}(H,\Aut(G)) = 1$ if and only if $H = L(H,\Aut(G))$. Therefore, we consider $H$ to be a subgroup of $G$ such that $H \ne L(H,\Aut(G))$ throughout the paper.
\section{Some upper bounds}
In this section we obtain several upper bounds for $\Pr(H,\Aut(G))$. We begin with the following result.

\begin{proposition}
Let $H$ and $K$ be two subgroups of a finite group $G$ such that $H \subseteq K$. Then
\[
\Pr(H,\Aut(G))\leq |K : H|\Pr(K,\Aut(G)).
\]
The equality holds if and only if $H = K$.
\end{proposition}
\begin{proof}
By  \eqref{formula2}, we have
\begin{align*}
|H||\Aut(G)|\Pr(H,\Aut(G))&= \underset{x\in H}{\sum}|C_{\Aut(G)}(x)|\\
&\leq \underset{x\in K}{\sum}|C_{\Aut(G)}(x)|\\
&=|K||\Aut(G)|\Pr(K,\Aut(G)).
\end{align*}
Hence, the result follows.
\end{proof}
\noindent As a corollary, we have the following result.
\begin{corollary}
Let $H$ be a subgroup of a finite group $G$. Then
\[
\Pr(H,\Aut(G))\leq |G : H|\Pr(G,\Aut(G))
\]
with  equality if and only if $H = G$.
\end{corollary}

\begin{theorem}\label{thm3.6}
Let $H$ be a subgroup of a finite group $G$ and $p$ the smallest prime dividing $|\Aut(G)|$.  Then
\[
\Pr(H,\Aut(G)) \leq \frac {(p - 1)|L(H,\Aut(G))| + |H|}{p|H|} - \frac {|X_H|(|\Aut(G)| - p)}{p|H||\Aut(G)|},
\]
where $X_H = \{x \in H : C_{\Aut(G)}(x) = \{I\}\}$ and $I$ is the identity automorphism of $G$.
\end{theorem}
\begin{proof}
We have   $X_H \cap L(H,\Aut(G)) = \phi$. Therefore
\begin{align*}
\underset{x\in H}{\sum}|C_{\Aut(G)}(x)| = & \, |X_H| + |\Aut(G)||L(H,\Aut(G))| \\
& + \underset{x\in H\setminus (X_H\cup L(H,\Aut(G)))}{\sum}|C_{\Aut(G)}(x)|.
\end{align*}

For $x  \in H \setminus (X_H \cup L(H,\Aut(G)))$ we have $C_{\Aut(G)}(x) \lneq  \Aut(G)$ which implies $|C_{\Aut(G)}(x)|\leq \frac {|\Aut(G)|}{p}$. Therefore
\begin{align}\label{our_bound2}
\underset{x\in H}{\sum}|C_{\Aut(G)}(x)| \leq & |X_H| + |\Aut(G)||L(H,\Aut(G))| \nonumber\\
& + \frac{|\Aut(G)|(|H| - |X_H| - |L(H,\Aut(G))|)}{p}.
\end{align}
Hence, the result follows from  \eqref{formula2} and \eqref{our_bound2}.
\end{proof}
We would like to mention here  that the upper bound obtained in Theorem \ref{thm3.6} is better than the upper bound obtained in \cite[Theorem 2.3 (i)]{moga}.
We also have  the following improvement of \cite[Corollary 2.2]{moga}.
\begin{corollary}\label{bound_like3/4}
Let $H$ be a subgroup of a finite group $G$. If $p$ and $q$ are the smallest primes dividing  $|\Aut(G)|$ and $|H|$ respectively then
\[
\Pr(H,\Aut(G)) \leq \frac{p + q - 1}{pq}.
\]
In particular, if $q \geq p$ then $\Pr(H,\Aut(G)) \leq \frac{2p - 1}{p^2} \leq \frac{3}{4}$.
\end{corollary}
\begin{proof}
Since $H \ne L(H,\Aut(G))$ we have $|H : L(H,\Aut(G))| \geq q$. Therefore, by Theorem \ref{thm3.6}, we have
\[
\Pr(H,\Aut(G)) \leq \frac{1}{p}\left(\frac{p - 1}{|H : L(H,\Aut(G))|} + 1\right) \leq \frac{p + q - 1}{pq}.
\]
\end{proof}
\noindent Further, if $H$ is a non-abelian subgroup of $G$ then we have the following result.
\begin{corollary}\label{bound_like5/8}
Let $H$ be a non-abelian subgroup of a finite group $G$. If  $p$ and $q$ are the smallest primes dividing  $|\Aut(G)|$ and $|H|$ respectively  then
\[
\Pr(H,\Aut(G)) \leq \frac{q^2 + p - 1}{pq^2}.
\]
In particular, if $q \geq p$ then $\Pr(H,\Aut(G)) \leq \frac{p^2 + p - 1}{p^3} \leq \frac{5}{8}$.
\end{corollary}
\begin{proof}
Since $H$ is non-abelian we have $|H : L(H,\Aut(G))| \geq q^2$. Therefore, by Theorem \ref{thm3.6}, we have
\[
\Pr(H,\Aut(G)) \leq \frac{1}{p}\left(\frac{p - 1}{|H : L(H,\Aut(G))|} + 1\right) \leq \frac{q^2 + p - 1}{pq^2}.
\]
\end{proof}

Now we obtain some characterizations of a subgroup $H$ of a finite group $G$ if equality holds in Corollary \ref{bound_like3/4} and Corollary \ref{bound_like5/8}. 

\begin{theorem}\label{prop4.1}
Let $H$ be a subgroup of a finite group $G$. If $p$ and $q$ are two primes such that  $\Pr(H,\Aut(G)) = \frac {p + q -1}{pq}$ then  $pq$  divides $|H||\Aut(G)|$. Further, if $p$ and $q$ are the smallest primes dividing $|\Aut(G)|$ and $|H|$ respectively then
\[
\frac{H}{L(H,\Aut(G))} \cong \mathbb Z_q.
\]
\end{theorem}
\begin{proof}
By \eqref{GenAutComDeg}, we have $(p + q -1)|H||\Aut(G)| = pq|\{(x,\alpha)\in H\times \Aut(G): \alpha(x) = x\}|$. Therefore, $pq$ divides $|H||\Aut(G)|$.

If $p$ and $q$ are the smallest primes dividing $|\Aut(G)|$ and $|H|$ respectively then, by Theorem \ref{thm3.6}, we have
\[
\frac{p + q -1}{pq} \leq \frac{1}{p}\left(\frac{p - 1}{|H : L(H,\Aut(G))|} + 1\right)
\]
which gives $|H : L(H,\Aut(G))| \leq q$. Hence, $\frac{H}{L(H,\Aut(G))} \cong \mathbb Z_q$.
\end{proof}
\noindent It is worth mentioning here that Theorem \ref{prop4.1} is a generalization of \cite[Theorem 2.4]{moga}.

\begin{theorem}\label{prop4.2}
Let $H$ be a  non-abelian subgroup of a finite group $G$. If $p$ and $q$ are two primes such that $\Pr(H,\Aut(G)) = \frac {q^2 + p - 1}{pq^2}$ then $pq$ divides  $|H||\Aut(G)|$. Further, if $p$ and $q$ are the smallest primes dividing $|\Aut(G)|$ and $|H|$ respectively then
\[
\frac{H}{L(H,\Aut(G))}\cong\mathbb Z_q \times \mathbb Z_q.
\]
In particular, if $H$ and $\Aut (G)$ are of even order and $\Pr(H,\Aut(G)) = \frac{5}{8}$ then $\frac{H}{L(H,\Aut(G))} \cong \mathbb Z_2 \times \mathbb Z_2$.
\end{theorem}

\begin{proof}
By \eqref{GenAutComDeg}, we have $(q^2 + p -1)|H||\Aut(G)| = pq^2|\{(x,\alpha)\in H\times \Aut(G): \alpha(x) = x\}|$. Therefore, $pq$ divides $|H||\Aut(G)|$.

If $p$ and $q$ are the smallest primes dividing $|\Aut(G)|$ and $|H|$ respectively then, by Theorem \ref{thm3.6}, we have
\[
\frac{q^2 + p -1}{pq^2} \leq \frac{1}{p}\left(\frac{p - 1}{|H : L(H,\Aut(G))|} + 1\right)
\]
which gives $|H : L(H,\Aut(G))| \leq q^2$. Since $H$ is non-abelian, $|H : L(H,\Aut(G))| \neq 1, q$. Hence, $\frac{H}{L(H,\Aut(G))} \cong \mathbb Z_q \times \mathbb Z_q$.
\end{proof}


\noindent The following  result gives partial converses of Theorems \ref{prop4.1} and \ref{prop4.2} respectively.

\begin{proposition}
Let $H$ be a subgroup of a finite group $G$. Let $p, q$ be the smallest prime divisors of $|\Aut(G)|$, $|H|$ respectively and $|\Aut(G) : C_{\Aut(G)}(x)| = p$ for all $x \in H \setminus L(H,\Aut(G))$.
\begin{enumerate}
\item
If $\frac{H}{L(H,\Aut(G))} \cong \mathbb Z_q$ then $\Pr(H,\Aut(G)) = \frac {p + q - 1}{pq}$.
\item
If $\frac{H}{L(H,\Aut(G))} \cong\mathbb Z_q \times \mathbb Z_q$  then $\Pr(H,\Aut(G)) = \frac {q^2 + p - 1}{pq^2}$.
\end{enumerate}
\end{proposition}
\begin{proof}
Since $|\Aut(G) : C_{\Aut(G)}(x)| = p$ for all $x \in H \setminus L(H,\Aut(G))$ we have $|C_{\Aut(G)}(x)| = \frac{|\Aut(G)|}{p}$ for all $x \in H \setminus L(H,\Aut(G))$. Therefore, by \eqref{formula2}, we have

\begin{align*}
\Pr(H,\Aut(G)) &= \frac{|L(H,\Aut(G))|}{|H|} + \frac {1}{|H||\Aut(G)|} \underset{x \in H \setminus L(H,\Aut(G))}{\sum}|C_{\Aut(G)}(x)|\\
&= \frac {|L(H,\Aut(G))|}{|H|} + \frac{|H| - |L(H,\Aut(G))|}{p|H|}.
\end{align*}
Thus
\begin{equation}\label{conv_eq}
\Pr(H,\Aut(G)) = \frac{1}{p}\left(\frac{p - 1}{|H : L(H,\Aut(G))|}  + 1\right).
\end{equation}
Hence, the results follow from \eqref{conv_eq}. 
\end{proof}

Note that if we replace $\Aut(G)$ by $\Inn(G)$, the inner automorphism group of $G$, then from  \eqref{GenAutComDeg}, we have  $\Pr(H,\Inn(G)) = \Pr(H,G)$ where
\[
\Pr(H,G) = \frac{|\{(x, y) \in H \times G : xy = yx\}|}{|H||G|}.
\]
Various properties of the ratio $\Pr(H,G)$ are studied in \cite{erl} and \cite{nY15}. We conclude this section showing that  ${\Pr}(H,\Aut(G))$ is bounded by  ${\Pr}(H,G)$.
\begin{proposition}
Let $H$ be a subgroup of a finite group $G$. Then
\[
\Pr(H,\Aut(G)) \leq \Pr(H,G).
\]
\end{proposition}
\begin{proof}
By  \cite[Lemma 1]{nY15}, we have
\begin{equation}\label{newbound01}
\Pr(H, G) =   \frac {1}{|H|}\underset{x\in H}{\sum}\frac {1}{|\cl_G(x)|}
\end{equation}
where $\cl_G(x) = \{\alpha(x) : \alpha \in \Inn(G)\}$. Since $\cl_G(x) \subseteq \orb(x)$ for all $x \in H$, the result follows from \eqref{grpdefn} and \eqref{newbound01}.
\end{proof}

\section{Some lower bounds}
We begin with the following lower bound.
\begin{theorem}\label{newthm3.6}
Let $H$ be a subgroup of a finite group $G$ and $p$ the smallest prime dividing $|\Aut(G)|$.  Then
\[
\Pr(H,\Aut(G)) \geq \frac {|L(H,\Aut(G))|}{|H|} + \frac {p(|H| - |X_H| - |L(H,\Aut(G))|) + |X_H|}{|H||\Aut(G)|},
\]
where $X_H = \{x \in H : C_{\Aut(G)}(x) = \{I\}\}$ and $I$ is the identity automorphism of $G$.
\end{theorem}
\begin{proof}
We have   $X_H \cap L(H,\Aut(G)) = \phi$. Therefore
\begin{align*}
\underset{x\in H}{\sum}|C_{\Aut(G)}(x)| = & \, |X_H| + |\Aut(G)||L(H,\Aut(G))| \\
& + \underset{x\in H\setminus (X_H\cup L(H,\Aut(G)))}{\sum}|C_{\Aut(G)}(x)|.
\end{align*}

For $x  \in H \setminus (X_H \cup L(H,\Aut(G)))$ we have $\{I\}\lneq C_{\Aut(G)}(x)$ which implies $|C_{\Aut(G)}(x)| \geq p$. Therefore
\begin{align}\label{our_bound1}
\underset{x\in H}{\sum}|C_{\Aut(G)}(x)| \geq & |X_H| + |\Aut(G)||L(H,\Aut(G))|\nonumber \\
& + p(|H| - |X_H| - |L(H,\Aut(G))|)
\end{align}
Hence, the result follows from  \eqref{formula2} and \eqref{our_bound1}.
\end{proof}

Now we obtain two lower bounds analogous to the lower bounds obtained in  \cite[Theorem A]{nY15} and \cite[Theorem 1]{ND10}.
\begin{theorem}\label{prop3.8}
Let $H$ be a subgroup of a finite group $G$.  Then
\[
\Pr(H,\Aut(G)) \geq \frac {1}{|S(H,\Aut(G))|}\left(1 + \frac {|S(H,\Aut(G))| - 1}{|H : L(H,\Aut(G))|}\right).
\]
The equality holds if and only if $\orb(x) = xS(H,\Aut(G))\,
 \forall \, x \in H \setminus L(H,\Aut(G))$.
\end{theorem}
\begin{proof}
For all $x \in H \setminus L(H,\Aut(G))$ we have $\alpha (x) = x[x, \alpha] \in xS(H,\Aut(G))$. Therefore $\orb(x) \subseteq xS(H,\Aut(G))$ and so $|\orb(x)| \leq |S(H,\Aut(G))|$
 for all $x \in H \setminus L(H,\Aut(G))$. Now, by  \eqref{grpdefn}, we have
\begin{align*}
\Pr(H,\Aut(G))& = \frac {1}{|H|}\left(\underset{x \in L(H,\Aut(G))}{\sum}\frac {1}{|\orb(x)|} + \underset{x \in H \setminus L(H,\Aut(G))}{\sum}\frac {1}{|\orb(x)|}\right)\\
&\geq \frac {|L(H,\Aut(G))|}{|H|} + \frac {1}{|H|}\underset{x\in H\setminus L(H,\Aut(G))}{\sum}\frac{1}{|S(H,\Aut(G))|}.
\end{align*}
Hence, the result follows.
\end{proof}

\begin{corollary}\label{lastcor}
Let $H$ be a subgroup of a finite group $G$. Then
\[
{\Pr}(H,\Aut(G))\geq \frac {1}{|[H,\Aut(G)]|}\left(1 + \frac {|[H,\Aut(G)]| - 1}{|H : L(H,\Aut(G))|}\right).
\]
\end{corollary}
\begin{proof}
For any two integers  $m \geq n$, we have
\begin{equation}\label{lemma4.4}
\frac {1}{n}\left(1 + \frac{n - 1}{|H : L(H,\Aut(G))|}\right) \geq \frac {1}{m}\left(1 + \frac{m - 1}{|H : L(H,\Aut(G))|}\right).
\end{equation}
Now, the result follows from Theorem \ref{prop3.8} and  \eqref{lemma4.4} noting that 
\[
|[H,\Aut(G)]|  \geq |S(H, \,\Aut(G))|.
\]
\end{proof}
\noindent Note that Corollary \ref{lastcor} is a generalization of  \cite[Equation (3)]{aK16}. 
Also
\begin{align*}
 \frac {1}{|[H,\Aut(G)]|}\left(1 + \frac {|[H,\Aut(G)]| - 1}{|H : L(H,\Aut(G))|}\right)
\geq & \frac
{|L(H,\Aut(G))|}{|H|} \\
& + \frac {p(|H| - |L(H,\Aut(G))|)}{|H||\Aut(G)|}.
\end{align*}
Hence, Corollary \ref{lastcor} gives better lower bound than  the lower bound obtained in \cite[Theorem 2.3 (i)]{moga}.
We conclude this section with the following generalization of \cite[Proposition 3]{aK16} which gives several equivalent conditions for  equality in Corollary \ref{lastcor}.
\begin{proposition}
  If $H$ is a subgroup of a finite group $G$ then the following statements are equivalent.
  \begin{enumerate}
    \item ${\Pr}(H,\Aut(G))= \frac {1}{|[H,\Aut(G)]|}\left(1 + \frac {|[H,\Aut(G)]| - 1}{|H : L(H,\Aut(G))|}\right).$

    \item $|{\orb}(x)| = |[H,\Aut(G)]|$ for all $x\in H\setminus L(H,\Aut(G))$.
    \item ${\orb}(x) = x[H,\Aut(G)]$ for all $x\in H\setminus L(H,\Aut(G))$, and hence $[H, \Aut(G)]$ $\subseteq L(H,\Aut(G))$.
    \item $C_{\Aut(G)}(x) \lhd \Aut(G)$ and $\frac{\Aut(G)}{C_{\Aut(G)}(x)} \cong [H, \Aut(G)]$ for all $x\in H\setminus L(H,\Aut(G))$.
    \item $[H, \Aut(G)] = \{x^{-1}\alpha(x) : \alpha \in\Aut(G)\}$ for all $x\in H\setminus L(H,\Aut(G))$.
  \end{enumerate}
\end{proposition}

\begin{proof}
First note that for all $x \in H$
\begin{equation}\label{lasteq1}
\orb(x) \subseteq x[H, \Aut(G)].
\end{equation}
Suppose that (a) holds. Then, by \eqref{grpdefn}, we have 
\[
\underset{x\in H \setminus L(H, \Aut(G))}{\sum}\left(\frac{1}{|\orb(x)|} - \frac{1}{|[H,\Aut(G)]|}\right) = 0.
\]
Now using \eqref{lasteq1}, we get (b). Also, if (b) holds then from \eqref{grpdefn}, we have (a). Thus (a) and (b) are equivalent.

  Suppose that (b) holds. Then for all $x\in H\setminus L(H,\Aut(G))$ we have $|{\orb}(x)| = |x[H,\Aut(G)]|$. Hence, using \eqref{lasteq1} we get (c).
 If $[H, \Aut(G)] \nsubseteq L(H,\Aut(G))$ then there exist $y \in [H, \Aut(G)]\setminus L(H, \Aut(G))$. Therefore  $\orb(y) = y[H, \Aut(G)] = [H, \Aut(G)]$, a contradiction. Hence $[H, \Aut(G)] \subseteq L(H,\Aut(G))$. It can be seen that  the map $f:\Aut(G)\to [H, \Aut(G)]$ given by $\alpha \mapsto x^{-1}\alpha(x)$, where $x$ is a fixed element of $H \setminus L(H, \Aut(G))$, is a surjective homomorphism with kernel $C_{\Aut(G)}(x)$. Therefore (d) follows. 
Since  $|\Aut(G)|/|C_{\Aut(G)}(x)| = |\orb(x)|$ for all $x\in H\setminus L(H,\Aut(G))$ we have (b). 
Thus (b), (c) and  (d) are equivalent.
 
The  equivalence of (c) and (e) follows from the fact that
$\orb(x) = x[H, \Aut(G)]$ if and only if $x^{-1}\orb(x) = [H, \Aut(G)]$ for all $x\in H\setminus L(H,\Aut(G))$. 
This completes the proof.
\end{proof}


\section{Autoisoclinism between pairs of groups}
The concept of isoclinism between two groups was introduced by Hall \cite{pH40} in the year 1940. After many years,  in 2013, Moghaddam et al. \cite{msE13} have introduced autoisoclinism between two groups. Recall that two groups $G_1$ and $G_2$ are said to be autoisoclinic  if there exist isomorphisms $\psi : \frac{G_1}{L(G_1)} \to \frac{G_2}{L(G_2)}$, $\gamma : \Aut(G_1) \to \Aut(G_2)$ and $\beta : [G_1, \Aut(G_1)] \to [G_2, \Aut(G_2)]$ such that the following diagram commutes
\begin{center}
$
\begin{CD}
   \frac{G_1}{L(G_1)} \times \Aut(G_1) @>\psi \times \gamma>> \frac{G_2}{L(G_2)} \times
   \Aut(G_2)\\
   @VV{a_{(G_1, \Aut(G_1))}}V  @VV{a_{(G_2, \Aut(G_2))}}V\\
   [G_1, \Aut(G_1)] @> \beta >> [G_2, \Aut(G_2)]
\end{CD}
$
\end{center}
where the maps $a_{(G_i, \Aut(G_i))}: \frac{G_i}{L(G_i)} \times \Aut(G_i) \to [G_i, \Aut(G_i)]$ for  $i = 1, 2$ are given by
\[
a_{(G_i, \Aut(G_i))}(x_iL(G_i), \alpha_i) = [x_i, \alpha_i].
\]
Such a triple $(\psi, \gamma, \beta)$ is called an autoisoclinism between $G_1$ and $G_2$.
We  generalize the notion of autoisoclinism between two groups in the following definition.
\begin{definition}
Let $H_1$ and $H_2$ be  two subgroups of the groups $G_1$ and $G_2$ respectively. A pair of groups  $(H_1, G_1)$ is said to be autoisoclinic to another pair of groups $(H_2, G_2)$   if there exist isomorphisms $\psi : \frac{H_1}{L(H_1, \Aut{G_1})} \to \frac{H_2}{L(H_2, \Aut(G_2))}$, $\gamma : \Aut(G_1) \to \Aut(G_2)$ and $\beta :
[H_1, \Aut(G_1)] \to [H_2, \Aut(G_2)]$  such that the following diagram commutes
\begin{center}
$
\begin{CD}
   \frac{H_1}{L(H_1, \Aut(G_1))} \times \Aut(G_1) @>\psi \times \gamma>> \frac{H_2}{L(H_2, \Aut(G_2))} \times
   \Aut(G_2)\\
   @VV{a_{(H_1, \Aut(G_1))}}V  @VV{a_{(H_2, \Aut(G_2))}}V\\
   [H_1, \Aut(G_1)] @> \beta >> [H_2, \Aut(G_2)]
\end{CD}
$
\end{center}
\noindent where the maps $a_{(H_i, \Aut(G_i))}: \frac{H_i}{L(H_i, \Aut(G_i))} \times \Aut(G_i) \to [H_i, \Aut(G_i)]$ for $i = 1,
2$ are given by
\[
a_{(H_i, \Aut(G_i))}(x_iL(H_i, \Aut(G_i)), \alpha_i) = [x_i, \alpha_i].
\]
Such a triple $(\psi, \gamma, \beta)$ is said to be an autoisoclinism between the pairs $(H_1, G_1)$
and $(H_2, G_2)$.
\end{definition}
We conclude this section with the following generalization of \cite[Lemma 2.5]{Rismanchian15}.
\begin{theorem}
Let $G_1$ and $G_2$ be two finite groups with subgroups $H_1$ and $H_2$ respectively. If $(\psi, \gamma, \beta)$ is an autoisoclinism between the pairs $(H_1, G_1)$ and
$(H_2, G_2)$ then
\[
\Pr(H_1, \Aut(G_1)) = \Pr(H_2, \Aut(G_2)).
\]
\end{theorem}

\begin{proof}
Consider the sets $\mathcal{S} = \{(x_1L(H_1,\Aut(G_1)),\alpha_1)\in \frac {H_1}{L(H_1,\Aut(G_1))}\times
\Aut(G_1):\alpha_1(x_1) = x_1\}$ and $\mathcal{T} = \{(x_2L(H_2,\Aut(G_2)),\alpha_2)\in \frac
{H_2}{L(H_2,\Aut(G_2))}\times \Aut(G_2):\alpha_2(x_2) = x_2\}$. Since  $(H_1, G_1)$ is
autoisoclinic to $(H_2, G_2)$ we have $|\mathcal{S}| = |\mathcal{T}|$. Again, it is clear that
\begin{equation}\label{auto-eq5.1}
|\{(x_1,\alpha_1)\in H_1\times \Aut(G_1):
\alpha_1(x_1) = x_1\}|=|L(H_1,\Aut(G_1))||\mathcal{S}|
\end{equation}
and
\begin{equation}\label{auto-eq5.2}
|\{(x_2,\alpha_2)\in H_2\times \Aut(G_2):
\alpha_2(x_2) = x_2\}|=|L(H_2,\Aut(G_2))||\mathcal{T}|.
\end{equation}
Hence, the result follows from \eqref{GenAutComDeg}, \eqref{auto-eq5.1} and \eqref{auto-eq5.2}.
\end{proof}

\section*{Acknowledgment}
The first author would like to thank  DST for the INSPIRE Fellowship.

\end{document}